\documentclass[a4paper,twoside]{article}
\usepackage{amsmath,amsfonts,amssymb,amsthm}
\usepackage{enumitem}
\usepackage[bookmarks]{hyperref}
\usepackage[center]{titlesec}
\usepackage{
mathrsfs}
\usepackage{tocbibind}
\usepackage{geometry}
\usepackage{fancyhdr}
\newtheorem{thm}{Theorem}[section]
\newtheorem{cj}{Conjecture}[section]
\newtheorem{defi}[thm]{Definition}
\newtheorem{lem}[thm]{Lemma}
\newtheorem{cor}[thm]{Corollary}
\newtheorem{pro}[thm]{Proposition}

\def\bb{\mathbb}
\def\ca{\mathcal}
\def\fr{\mathfrak}

\def\bitm{\begin{itemize}} \def\eitm{\end{itemize}}
\def\benu{\begin{enumerate}} \def\eenu{\end{enumerate}}

\def\bpf{\begin{proof}}\def\epf{\end{proof}}
\def\beq{\begin{equation}}\def\eeq{\end{equation}}

\def\beqs{\begin{eqnarray}}\def\eeqs{\end{eqnarray}}
\def\beqsnl{\begin{eqnarray*}}\def\eeqsnl{\end{eqnarray*}}

\begin{document}

\title{\textbf{
Restriction Theorems on M\'{e}tiver Groups Associated to Joint Functional Calculus
}\footnote{Date: Jun. 5, 2013. Last revised on Oct. 25, 2014. }}
\author{Heping Liu, An Zhang \footnote{Corresponding author. The authors are supported by National Natural Science Foundation of China under Grant \#11371036 and the Specialized
Research Fund for the Doctoral Program of Higher Education of China under Grant \#2012000110059. Contact Information: Department of Mathematics, School of Mathematical Science, Peking University, No.5 Yiheyuan St, Haidian, Beijing, 100871, China. E-mails: \texttt{hpliu@math.pku.edu.cn} (H. Liu), \texttt{anzhang@pku.edu.cn} (A. Zhang, for Correspondence).}}
\date{}
\maketitle

\begin{abstract}
 In this article, we get on M\'{e}tivier groups the spectral resolution of a class of operators $m(\mathcal{L}, -\Delta_\fr{z})$, the joint functional calculus of the sub-Laplacian and Laplacian on the center. Then, we give some restriction theorems, asserting the mix-norm boundness of the spectral projection operators $\mathcal{P}_{\mu}^{m}$ for two classes of functions $m(a,b)= (a^\alpha+b^\beta)^\gamma$ or $(1+a^\alpha+b^\beta)^\gamma$, with $\alpha, \beta>0, \gamma\neq0$.
\end{abstract}

\section{Introduction}\label{introduction}
In this paper, we extend the mix-norm boundness obtained by V. Casarino and P. Ciatti \cite{cc12} to general two classes of projection operators (we call \emph{projector} for simplicity) on M\'{e}tivier groups (a class of 2-step nilpotent Lie group, first defined and studied by G. M\'{e}tivier in \cite{metivier80}). M\'{e}tivier group class is strictly more general than the H-type group class introduced by A. Kaplan \cite{kaplan80}, with the Heisenberg group being the \emph{only} special one (of H-type) with $1$-dimensional center.

The ``restriction-type" operator (spectral projector) we study acts on the central variables by the Euclidean Fourier transform while acting on the ``space-$v$" variables by the spectral projection of the twisted Laplacian. As the quotient of a M\'{e}tiver group corresponding to the hyperplanes in the center is isomorphic to the Heisenberg group, we can use the spectral projector on the Heisenberg group to estimate that on the M\'{e}tivier group by a ``partial" Radon transform. For the Fourier transform on the central variables, we use the famous Tomas-Stein theorem. Our result includes not only \emph{homogeneous} (like $\mathcal{L}^2-\Delta_\fr{z}$) but also \emph{inhomogeneous} operators (like full Laplacian $\Delta_G=\mathcal{L}-\Delta_\fr{z}$), and also cover the uniform-norm boundness $L^p\rightarrow L^{p'}$ with  exponents in the corresponding range.

First we recall some histories. The restriction problem on $\mathbb{R}^n$, denoted by $R_S(p\rightarrow q)$, cares the $L^p(\mathbb{R}^n)\rightarrow L^q(S)$  boundness of the Fourier transform for any hypersurface $S$ with boundary endowed with the Lebesgue surface measure $d\sigma$. It has many useful applications in both harmonic analysis and PDE. The restriction probolem with respect to the unit sphere (more generally for any compact hypersurface with boundary and non-vanishing Gaussian curvature everywhere) is given in the following conjecture.
\begin{cj}\label{c1}\emph{[Stein's Restriction Conjecture]}
\[R_{\bb{S}^{n-1}}(p\rightarrow q) \mbox{~ holds if and only if~~ }p <\frac{2n}{n+1} \text{ and~ } q\le \frac{n-1}{n+1} p'.\]
\end{cj}
The conjecture has been proved for $q=2$ or $n=2$ by P. Tomas, E. Stein and C. Fefferman. For $n=3$, J. Bourgain an L. Guth have recently gotten the best result so far in \cite{bg11}, where they proved the dual extension theorem $E_{S^{n-1}}(\infty\rightarrow p')$ for $p'>3\frac{3}{10}$. The authors used the method of multilinear theory from \cite{bct06} together with the Kakeya maximal estimate due to T. Wolff \cite{wolff95improvedboundforkakeya} and improved a bit an old result of T. Tao (Tao's bilinear approach gives $p'>3\frac{1}{3}$). In a word, some exciting progresses have been made while the whole picture is still far from known. BIG OPEN!

Recalling the famous Tomas-Stein Theorem, due to P. Tomas and E. Stein \cite{tomas75} (Stein's result for the endpoint is unpublished), they proved Conjecture \ref{c1} for $q=2$. It corresponds to the $L^p\rightarrow L^{p'}$ boundness of the convolution operator $f*\widehat{d\sigma_\lambda}$, which is just the spectral projector of the positive Laplacian $-\Delta$ whose symbol is  $|\xi|^2$. In \cite{str89}, R. Strichartz study analogues of this in other settings from a viewpoint of harmonic analysis as spectral theory of sub-Laplacians. Motivated by this idea, D. M\"{u}ller \cite{mul90} proved a mix-norm\footnote{We will use the same mix-norm through this paper: the $(p,q)$-type mix-norm of a function $f$ on the measure space $X\times Y$ is defined by $\|f\|_{L^p_xL^q_y}=\Big(\int_Y\big(\int_X |f(x,y)|^p dx\big)^{\frac{q}{p}}dy\Big)^{\frac{1}{q}}$, which  is reduced to the $L^p$ norm $\|f\|_{L^p(X\times Y)}$ when $p=q$.} boundness $L^\infty_t L^p_z\rightarrow L^1_t L^{p'}_z$ of the restriction operator associated to the sub-Laplacian on the Heisenberg group, using a bound of the spectral projector of the twisted Laplacian $\|\Lambda_k g\|_{L^{p'}} \lesssim k^{(n-1)(1-\frac{2}{p'})}\|g\|_{L^p}$. The exponent on the center is trivial because of the trivial $1$-dimensional Tomas-Stein theorem. So when center dimension is bigger than one, it's reasonable to get restriction theorems for exponents of wider range. See \cite{tha91,tha98,lw11,ls13} for related results. In \cite{cc12}, Casarino and Ciatti used an improved sharp bound of $\|\Lambda_k\|_{L^p\rightarrow L^q}$ to get a greatly improved mix-norm bound for the restriction operators associated to the sub-Laplacian and full Laplacian on the Heisenberg group and simultaneously an analogue for sub-Laplacian on the M\'{e}tivier group. More precisely, on a M\'{e}tivir group with dimension $2n+d$, where $d$ is the dimension of the center, we have the following theorem of Casarino and Ciatti about the restriction operator $\ca{P}_\lambda^\ca{L}$ associtaed to the sub-Laplacian $\ca{L}$ (details will be explained later).
\begin{thm}
Given $1\le p\le 2\le q\le 2, 1\le r\le 2\frac{d+1}{d+3}$, we have \[\|\mathcal{P}_\mu^{\mathcal{L}}f\|_{L^{r'}_zL^q_v} \le C_\mu \|f\|_{L^r_zL^p_v},\] with \[C_\mu \lesssim \mu^{n(\frac{1}{p}-\frac{1}{q})+d(\frac{1}{r}-\frac{1}{r'})-1}.\]
\end{thm}
In \cite{ls13}, the first author advised to compute some joint functional calculus of the Laplacians including the full Laplacian, first done by M. Song, on groups of H-type. Combining the two works, we obtain the mix-norm bound in \cite{cc12} for the joint functional calculus in \cite{ls13}. Our main result is the following theorem about two classes of operators $m(\ca{L},-\Delta_{\fr{z}})$ from the joint functional calculus  of the sub-Laplacian $\ca{L}$ and Laplacian on the center $-\Delta_{\fr{z}}$.
\begin{thm}
Given $\alpha,\beta>0, \gamma\neq0, 1\le p\le 2\le q\le \infty, (d,p,q)\neq (1,2,2), 1\le r\le 2\frac{d+1}{d+3}$, $m_1(a,b)=(a^\alpha+b^\beta)^\gamma, m_2(a,b)=(1+a^\alpha+b^\beta)^\gamma$, then we have \[\|\mathcal{P}_\mu^{m_i} f\|_{L^{r'}_zL^q_v}\le C_\mu^{m_i} \|f\|_{L^r_zL^p_v},\quad i=1,2,\] with
\[C_\mu^{m_1}\lesssim
\left\{\begin{array}{cc}
   \mu^B \qquad & \mu^{\frac{1}{\gamma}(\frac{1}{\alpha}-\frac{1}{2\beta})}\leq 1,\\
   \mu^A \qquad & \mu^{\frac{1}{\gamma}(\frac{1}{\alpha}-\frac{1}{2\beta})}> 1.
\end{array}\right.\]
and
\[C_\mu^{m_2}\lesssim
\left\{\begin{array}{cll}
    \mu^B \qquad & \mu\rightarrow \infty^{sign\gamma}, &\mu^{\frac{1}{\gamma}(\frac{1}{\alpha}-\frac{1}{2\beta})}\leq 1,\\
    \mu^A \qquad & \mu\rightarrow \infty^{sign\gamma}, &\mu^{\frac{1}{\gamma}(\frac{1}{\alpha}-\frac{1}{2\beta})}> 1, \\
    |1-\mu|^D \qquad & \mu\rightarrow e^{\frac{sign\gamma}{\infty}}, &|1-\mu|^{\frac{1}{\alpha}-\frac{1}{2\beta}}\leq 1,\\
    |1-\mu|^C \qquad & \mu\rightarrow e^{\frac{sign\gamma}{\infty}}, &|1-\mu|^{\frac{1}{\alpha}-\frac{1}{2\beta}}> 1.
\end{array}\right.\] Details about the definition of restriction operator $\ca{P}_\mu^{m}$ and the constants $A,B,C,D$ dependent of $n,d, p,q,r,\alpha,\beta,\gamma$ will be given later by \emph{(\ref{restriction operator formula})} and \emph{(\ref{exponent})} in subsection \ref{main results}.

\end{thm}

We arrange the remaining part of our paper in the following order. In section \ref{main results}, we will give the main result without proof after introducing basic notations and the restriction operators by functional calculus on the M\'{e}tivier group. Following the theorems, some remarks will also be given. In section \ref{proof of main results}, we will put our attention to the detailed proofs of the main result given in section \ref{main results}.

Concerning the boundness we care in this note, we introduce two notations: we will always use ``$\lesssim$" to indicate the left side is less than a constant multiple \footnote{We don't care the detailed expression of the constant in the present formula.} of the right side, while ``$\lesssim_{\lambda,\ldots}$" to mean the constant is dependent of $\lambda,\ldots$; We also use ``$\sim$" for ``almost" equal, accurately, the quotient is bounded both from above and below.

\section{Main Results}\label{main results}
In subsection \ref{restriction operator}, we introduce the Laplacians on the M\'{e}tivier Group and give the definition and explicit formula of the restriction operators associated to the joint functional calculus of the sub-Laplacian and Laplacian on the central variables. By the partial Randon transform, we can connect our M\'{e}tivier group with the Heisenberg group and  then use the bound of spectral projector of the twisted Laplacian on the Heisenberg group (scaled special Hermite projector) to estimate the mix-norm bound of the restriction operators associated to the Laplacians on the M\'{e}tivier group. In subsection \ref{bound results}, we list two detailed main theorems, but leave their proofs in the next section.

\subsection{Restriction Operators on the M\'{e}tivier Group}\label{restriction operator}
First, we will give some definitions. We take many notations and terminologies from \cite{tha98,cc12}, to which the reader can refer if is not very familiar with them. Let $G$ be a connected, simply connected, two-step nilpotent Lie group, associated with Lie algebra $\mathfrak{g}$, endowed with an inner product $\langle\cdot,\cdot\rangle$. The Lie algebra $\fr{g}$ can be decomposed into the direct sum $\mathfrak{g}=\mathfrak{z}+\mathfrak{v}$, with the center $\fr{z}$ and its complement $\fr{v}$. Take $d = \dim\fr{z}$, $k = \dim\fr{v}$, both which we assume are always positive integers, and $\fr{z}^\ast$ denotes the dual of $\fr{z}$ with dual norm $|\cdot|$ induced by the inner product $\langle\cdot,\cdot\rangle$ in $\fr{z}$. The unit ball in the dual space $\fr{z}^\ast$ is denoted by $S =\{\omega \in \fr{z}^\ast, |\omega|=1\}$. For each $\omega \in S$, we can find a unique unit $Z_\omega \in \fr{z}$ such that $\omega(Z_\omega)=1$. Then the center $\fr{z}$ can be decomposed as $\fr{z}=\mathbb{R}Z_\omega+ker\omega$, with the quotient isomorphism $\fr{z}/ker\omega\simeq\mathbb{R}Z_\omega$. Denote $\fr{g}_\omega=\mathbb{R}Z_\omega+\fr{v}$, then we have isomorphism $\fr{g}/ker\omega\simeq\fr{g}_\omega$. As $ker\omega$ is an ideal of $\fr{z}$,  $\fr{g}_\omega$
is a Lie subalgebra. We denote by $G_\omega$ the connected simply connected Lie subgroup of $G$, associated with Lie algebra $\fr{g}_\omega$. We define the M\'{e}tivier property by the following non-degeneracy description.

\begin{defi}
Bilinear function $B_\omega(U,V)\stackrel{def}{=}\omega([U,V])$ with $U,V \in \fr{v}, \omega \in S $, is called non-degenerate, if
\beq\label{non-degeneracy}
B_\omega(U, V) = 0, ~\forall \,U \in \fr{v}\quad \Rightarrow \quad V = 0.\eeq
\end{defi}
\begin{defi}
Group $G$ discussed above is called a M\'{e}tivier Group, if $B_\omega$ is non-degenerate for all $\omega \in S$.
\end{defi}

D. M\"{u}ller and A. Seeger \cite{ms04singular} gave an example that is a M\'{e}tivier group but not of H-type. For completion, we explain it here: given Lie algebra $\fr{g}=\fr{v}+\fr{z}=\mathbb{R}^8+\mathbb{R}^2$, with the Lie bracket
\[[V+Z,U+W]=0+(v^tJ_{(1,0)}u,v^tJ_{(0,1)}u),\]
where $u,v$ are the coordinates of $U,V$ in an orthonormal basis, and matrix \[J_z=\left(\begin{array}{cc}
                   0 & E_z \\
                   -E_z & 0
                 \end{array}\right),\quad
E_z=\left(\begin{array}{cccc}
                                                         z_1 & 0 & 0 & -z_2 \\
                                                         z_2 & z_1 & 0 & 0 \\
                                                         0 & z_2 & z_1 & 0 \\
                                                         0 & 0 & z_2 & z_1
                                                       \end{array}\right).
\] As $|J_z|=(z_1^4+z_2^4)^2\neq0 \mbox{ unless } z=0$, we see $\fr{g}$ is M\'{e}tivier with non-degeneracy property (\ref{non-degeneracy}). Now we try to say that $\fr{g}$ is not of H-type. Actually, assume that there is another H-type Lie algebra $\fr{g}'=\fr{v}'+\fr{z}'$ (in abuse of notation, a map\footnote{See accurate definition of H-type and this map in Kaplan \cite{kaplan80}.} $J'_{z'}: \fr{v}'\rightarrow\fr{v}'$ just as $B_\omega(U)$ above, is orthogonal whenever $|z'|=1$, and we sometimes identify $\fr{z}$ and $\fr{z}^*$) and a Lie algebra isomorphism $\alpha: \fr{g}\rightarrow\fr{g}'$, then under orthonormal basis,
$\alpha=\left(\begin{array}{cc}
          A & 0 \\
          C & D
        \end{array}\right)$ (notice that $\alpha$ preserves the centers and induces an isomorphism $\fr{z}\rightarrow \fr{z}'$).
By the conservation of Lie brackets, for any $u,v,z$, we have \[(A^{-1}v)^tJ'_{D^tz}(A^{-1}u)= v^tJ_zu,\] which tells (from $J'^2_{z'}=-|z'|^2 I$) the following determinant relation  \[|D^tz|^8=|J'_{D^tz}|=|A|^2|J_z|=|A|^2(z_1^4+z_2^4)^2.\] Take $|A|^{-1/4}D^t=\left(\begin{array}{cc}
                                                                                    a & b \\
                                                                                    c & d
                                                                                  \end{array}\right)
$, then we have
\[[(az_1+bz_2)^2+(cz_1+dz_2)^2]^2=z_1^4+z_2^4,\]
which implies a contradiction
  \[\left\{\begin{array}{rcl}
       a^2+c^2&=&1, \\
       b^2+d^2&=&1, \\
       2(ab+cd)^2+1&=&0,\\
       ab+cd&=&0.
  \end{array}\right.\]
For M\'{e}tivier Group $G$, the dimension of $\fr{v}$ is even from the non-degeneracy and skew-symmetry, which we denote by $\dim\fr{v}= 2n$, then Lie subgroup $G_\omega \simeq \mathbb{H}^n$, the Heisenberg Group with Lie algebra $\fr{h}_n = \mathbb{C}^n+\bb{R}$. We will use the spectral decomposition of the Laplacians on the Heisenberg Goup to get its counterpart on the M\'{e}tivier Group, and finally obtain the corresponding restriction Operators.

By the nilpotency of $G$, we can parametrize $G$ by its Lie algebra $\fr{z}+\fr{v}$, through the surjective exponential map. Fix a basis of Lie algebra $\fr{g}$,
\[\{V_1, V_2,\ldots, V_{2n}; Z_1, Z_2,\ldots, Z_d\}, \]
then we can endow every point of  group $G$ with an exponential coordinate $(V,Z)$ or $(v,z) \in \mathbb{R}^{2n}\times\mathbb{R}^{d}$. By Baker-Campbell-Hausdorff formula, we get the multiplication law
\[(V,Z)\cdot(V',Z')=(V+V',Z+Z'+\frac{1}{2}[V,V']),\]
with $V,V' \in \fr{v}, Z,Z' \in \fr{z}$.  Simple computation gives the left-invariant vector fields
\begin{align*}
\widetilde{V_j}&=\frac{\partial}{\partial v_j}+\frac{1}{2}\sum_{i=1}^d  \langle Z_i,[V,V_j]\rangle \frac{\partial}{\partial z_i}, \qquad j=1,\ldots,2n, \\
\widetilde{Z_i}&=\frac{\partial}{\partial z_i}, \qquad \quad i=1,\ldots,d,
\end{align*}
associated respectively to one-parameter subgroups $\{(sV_j,0)|s \in \mathbb{R}\}$ and $\{(0,tZ_i)|t \in \mathbb{R}\}$. These $2n+d$ vector fields form a basis of the tangent boundle of $G$. Now we can define on $G$ the \emph{sub-Laplacian}, the \emph{Laplacian on the center}, and the \emph{full Laplacian} respectively to be
\[\mathcal{L}=-\sum_{j=1}^{2n} (\widetilde{V_j})^2,\quad -\Delta_\fr{z}=-\sum_{i=1}^d (\widetilde{Z_i})^2,\quad \Delta_G=\mathcal{L}-\Delta_\fr{z}.\] H\"{o}mander's theorem tells that the sub-Laplacian and full Laplacian are hypoelliptic, positive, and essentially self-adjoint, and the Laplacians play an important role in harmonic analysis on the group.

The \emph{partial} Radon transform on the central variables is a powerful tool for us to get the spectral decomposition on M\'{e}tivier group from that on the Heisenberg group.
\begin{defi}
For $\omega \in S, f \in \mathcal{S}(G)$, the Schwartz space on G, we define the partial Radon transform of $f$ to be
\[R_\omega f (V,t)= \int_{ker\omega} f(V,tZ_\omega+Z') d\sigma(Z').\]
\end{defi}
\begin{lem}\label{relations}
For Schwartz functions $f \in \mathcal{S}(G), g \in \mathcal{S}(G_\omega)$, we have the following formulas
\begin{enumerate}[label=\emph{(\arabic*)}]
  \item
  Sub-Laplacians on $G$ and $G_\omega$:
  \begin{align*}
  R_\omega (\widetilde{V_j} f) (V,t) = V_j^\omega (R_\omega f)(V,t), \quad
  R_\omega(\mathcal{L}f) =\mathcal{L}^\omega(R_\omega f) ,                                                                 \end{align*}
  where $V_j^\omega = \frac{\partial}{\partial v_j}+\frac{1}{2}\omega([V,V_j]) \frac{\partial}{\partial t}$  and $\mathcal{L}^\omega = -\sum_{j=1}^{2n} (V_j^\omega)^2$ are respectively the left-invariant vector field and sub-Laplacian on $G_\omega$.
  \item Sub-Laplacian and twisted Laplacian on $G_\omega$:
  \begin{align*}
   \fr{F}_1(V_j^\omega g)(V,\lambda) = V_j^{\lambda,\omega}(\fr{F_1}g)(V,\lambda), \quad
   \fr{F}_1(\mathcal{L}^\omega g) = L^{\lambda,\omega}(\fr{F}_1g),
  \end{align*}
  where $\fr{F}_1$ means the inverse Fourier transform on the one dimensional center, and $V_j^{\lambda,\omega}= \frac{\partial}{\partial v_j}-\frac{i\lambda}{2}\omega([V,V_j])$ and $L^{\lambda,\omega}= -\sum_{j=1}^{2n} (V_j^{\lambda,\omega})^2$ are respectively the  $\lambda$-twisted left-invariant vector field and $\lambda$-twisted Laplacian on $G_\omega$.
  \item Sub-Laplacian and $\lambda\omega$-twisted Laplacian on $G$:
  \begin{align*}
    \fr{F}_\fr{z}(\widetilde{V_j}f)(V,\lambda\omega) = V_j^{\lambda\omega}(\fr{F}_\fr{z}f)(V,\lambda\omega),\quad
    \fr{F}_\fr{z}(\mathcal{L}f) = L^{\lambda\omega}(\fr{F}_\fr{z}f),
  \end{align*}
  where $\fr{F}_\fr{z}$ means the inverse Fourier transform on the center, and $V_j^{\lambda\omega}= \frac{\partial}{\partial v_j}-\frac{i\lambda}{2}\omega([V,V_j])$ and $L^{\lambda\omega}= -\sum_{j=1}^{2n} (V_j^{\lambda\omega})^2$ are respectively the $\lambda\omega$-twisted left-invariant vector field and  $\lambda\omega$-twisted Laplacian on $G$.
\end{enumerate}
\end{lem}
\emph{Remark:}
Lemma \ref{relations} tells us that the $\lambda\omega$-twisted Laplacian on $G$ is nothing but the $\lambda$-twisted Laplacian on $G_\omega$. This give the idea of how to get the spectral projection of the twisted Laplacian on $G$.

Observing the non-degeneracy property (\ref{non-degeneracy}) of skew-symmetric bilinear function $B_\omega$, we can use an invertable linear transfrom $A_\omega$ to change the bilinear function to the standard symplectic form
$\left(
\begin{array}{cc}
0 & \mathbb{I}^n\\
-\mathbb{I}^n & 0\\
\end{array}
\right)$.
In this new coordinates, denoted e.g. by $\{y_j\}_{j=1}^{2n}$, the $\lambda$-twisted Laplacian on $G_\omega$ is then
\[L^{\lambda,\omega}= -\sum_{j=1}^{2n} \frac{\partial^2}{\partial y_j^2}+ \frac{\lambda^2}{4}\sum_{j=1}^{2n} y_j^2+i\lambda\sum_{j=1}^{n} \Big(y_j\frac{\partial}{\partial y_{j+n}}-y_{j+n}\frac{\partial}{\partial y_j}\Big),\] which is just the usual $\lambda$-twisted Laplacian $L^\lambda$ on the Heisenberg group $\mathbb{H}^n$ and actually we have educed an isomorphism (non-isometric) between $G_\omega$ and $\bb{H}^n$. Then we get the spectral decomposition of the $\lambda\omega$-twisted Laplacian of $G$ in the following proposition.

\begin{pro}\label{proposition}
For $g \in \mathcal{S}(\fr{v}), \omega \in S$, take $g_\omega= g\circ (A_\omega)^{-1}$ and denote by $\Pi_k^{\lambda\omega}$ the spectral projector of the $\lambda\omega$-twisted Laplacian on $G$, then
\begin{align}\label{special hermite expansion}
 g  = (2\pi)^{-n} \lambda^n \sum_k \Pi_k^{\lambda\omega}g,\quad
 \Pi_k^{\lambda\omega}g  = (\Lambda_k^\lambda g_\omega)\circ A_\omega,
\end{align}
where $\Lambda_k^\lambda$ is the spectral projection of the $\lambda$-twisted Laplacian $L^\lambda$ on the Heisenberg group $\mathbb{H}^n$, given by the twisted convolution $\Lambda_k^\lambda g (z)= g \times_\lambda \varphi_k^{|\lambda|} (z)$ for $z \in \mathbb{C}^n$, and the special Hermite function  $\varphi_k^\lambda(z) = L_k^{n-1}(\frac{\lambda}{2}|z|^2)e^{-\frac{\lambda }{4}|z|^2}$, where $L_k^{n-1}$ is the Laguerre polynomial of type $n-1$ and degree $k$.
\end{pro}

Now, we use the spectral projector $\Pi_k^{\lambda\omega}$ to give the restriction operators associated to the joint functional calculus of $\mathcal{L} \text{ and } -\Delta_\fr{z}$.
From the inverse Fourier transform formula on the central variables, polar coordinates transformation, and spectral expansion (\ref{special hermite expansion}), we can get the following expansion for $f\in \ca{S}(G)$
\begin{align}\label{f expansion}
f(V,Z)& \sim \int_{\bb{R}^d} e^{-i\eta(Z)}\fr{F}_\fr{z}f(V,\eta) d\eta\nonumber\\
& = \iint_{\bb{S}^{d-1}\times\bb{R}^+} e^{-i\lambda\omega(Z)}\fr{F}_\fr{z}f(V,\lambda\omega)\lambda^{d-1}d\sigma(\omega)d\lambda\nonumber\\
& \sim \int_0^\infty\Big(\sum_{k=0}^\infty \lambda^{n+d-1} \int_{S^{d-1}} e^{-i\lambda\omega(Z)} (\Pi_k^{\lambda\omega}\circ \fr{F}_\fr{z})f(V,\lambda\omega)\,d\sigma(\omega)\Big)d\lambda.
\end{align}
Since the $k$-term in the sum is the joint eigenfunction\footnote{See for $\ca{L}$ from (3) in Lemma \ref{relations} and Proposition \ref{proposition} and it's obvious for $-\Delta_{\fr{z}}$.}
of $\mathcal{L}$ and $-\Delta_\fr{z}$, associated to the spectrum ray $R_k=\big((2k+n)\lambda,\lambda^2\big)$, we can naturally define the functional calculus operator associated to a function $m$ by
\begin{align}\label{functional calculus}
&m(\mathcal{L},-\Delta_\fr{z})f(V,Z)\nonumber\\
&\sim \int_0^\infty\Big(\sum_{k=0}^\infty m\big((2k+n)\lambda,\lambda^2\big)\lambda^{n+d-1} \int_{S^{d-1}}e^{-i\lambda\omega(Z)} (\Pi_k^{\lambda\omega}\circ \fr{F}_\fr{z})f(V,\lambda\omega)\,d\sigma(\omega)\Big)d\lambda,
\end{align} given a ``proper" $m$ such that $m\big((2k+n)\lambda,\lambda^2\big)$ is differentiable, positive, and strictly monotonic on $\mathbb{R}^+$ with regards to $\lambda$.
Generally, given a spectral decomposition of operator $D$\[D = \int_{\mathbb{R}^+} \lambda dE_\lambda,\] then the associated \emph{restriction operator} (spectral projector) can be defined as
\begin{align*}
   \mathcal{P}_\lambda^D & =\lim_{\epsilon\rightarrow0}\frac{1}{\epsilon}\chi_{(\lambda-\epsilon,\lambda+\epsilon)}(D)\\
  & =\lim_{\epsilon\rightarrow0}\frac{1}{\epsilon}\int_{\lambda-\epsilon}^{\lambda+\epsilon} dE_\mu,
\end{align*} and similarly, the restriction operator associated to $m(D)$ can be defined as \[\lim_{\epsilon\rightarrow0}\frac{1}{\epsilon}\chi_{(\lambda-\epsilon,\lambda+\epsilon)}\big(m(D)\big).\]
Now, we will care the operator $m(\mathcal{L},-\Delta_\fr{z})$ on $G$. $~\forall\,\mu \in \mathbb{R}^+$, $\mu_k$ denotes the solution $\lambda$ of the equation $m((2k+n)\lambda,\lambda^2)=\mu$ and $\mu_k'$ denotes the derivative relative to $\mu$, and we will write simply $\mathcal{P}_\mu^m$ to mean the restriction operator $\mathcal{P}_\mu^{m(\mathcal{L},-\Delta_\fr{z})}$, whose formula is given in the following theorem.

\begin{thm}\label{t-spectral decomposition}
$\forall\, f \in \mathcal{S}(G)$, \,$f$ have the following expansion
\[f  =\int_{\mathbb{R}^+}  \mathcal{P}_\mu^m f d\mu\]
under the spectral decomposition $m(\ca{L},-\Delta_{\fr{z}})\circ\ca{P}_\mu^m=\mu~ \ca{P}_\mu^m$,
where $\ca{P}_\mu^m$ is the restriction operator defined by
\beq\label{restriction operator formula}
\mathcal{P}_\mu^m f (V,Z)  \sim \sum_{k=0}^\infty \mu_k^{n+d-1} \mu_k' \int_{S^{d-1}}e^{-i\mu_k\omega(Z)} (\Pi_k^{\mu_k\omega}\circ \fr{F}_\fr{z})f(V,\mu_k\omega)\,d\sigma(\omega).
\eeq
\end{thm}
\begin{proof}
By changing variables from (\ref{f expansion}) or (\ref{functional calculus}),
  \[f(V,Z) \sim \int_0^\infty\Big(\sum_{k=0}^\infty \mu_k^{n+d-1} \int_{S^{d-1}} e^{-i\mu_k\omega(Z)}(\Pi_k^{\mu_k\omega}\circ \fr{F}_\fr{z})f(V,\mu_k\omega)\,d\sigma(\omega)\Big)\,d\mu_k.  \]
From Lemma \ref{relations} and Proposition \ref{proposition}, we have, after taking \[f_{\mu_k\omega}(V,Z)=e^{-i\mu_k\omega(Z)}(\Pi_k^{\mu_k\omega}\circ \fr{F}_\fr{z})(V,\mu_k\omega),\] that  \begin{align*}\ca{L}f_{\mu_k\omega}=\fr{F}_{\fr{z}}^{-1}L^{\lambda \omega}\fr{F}_{\fr{z}}f_{\mu_k\omega}
=(2k+n)\mu_k~ f_{\mu_k\omega}, \quad
(-\Delta_{\fr{z}})f_{\mu_k\omega}=\mu_k^2~ f_{\mu_k\omega},
\end{align*} and
\begin{align*}
m(\ca{L},-\Delta_{\fr{z}})f_{\mu_k\omega}
= m\big((2k+n)\mu_k, \mu_k^2\big)~f_{\mu_k\omega}
= \mu~ f_{\mu_k\omega},
\end{align*} which finally gives
\[m(\ca{L},-\Delta_{\fr{z}})\ca{P}_\mu^mf=\mu~\ca{P}_\mu^mf.\]
\end{proof}

\subsection{Mix-Norm Boundness of the Restriction Operator $\mathcal{P}_\mu^m$}\label{bound results}
In \cite{cc12}, V. Casarino and P. Ciatti have given the mix-norm boundness of the restriction operator $\ca{P}_\mu^{\ca{L}}$ associated to the sub-Laplacian $\mathcal{L}$. We now use the similar method to get the mix-norm boundness of the restriction operators associated to a class of operators $m(\mathcal{L},-\Delta_\fr{z})$. In this subsection, we will state our main theorems of this paper in detail.

First give some notations.
We define a function $\phi$ on [0,1/2] by
\beq\label{phi}
\phi(s)=
\left\{ \begin{array}{cc}
   -s/2\quad & s\leq s^\ast\\
  ns-1/2\quad & s\geq s^\ast
\end{array} \right.,\text{ where } s^\ast = \frac{1}{2n+1},
\eeq from which we see both $\phi$ and $s^\ast$ are related to dimension $n$.
Given $\alpha, \beta \in \mathbb{R}^+, \gamma \in \mathbb{R}^\ast, 1\leq p\leq 2\leq q\leq\infty, 1\leq r\leq 2\frac{d+1}{d+3}$, we define four numbers relative to $(\alpha,\beta,\gamma,p,q,r,n,d)$:
\begin{align}\label{exponent}
  A=&~\frac{1}{2\beta\gamma}[n(\frac{1}{p}-\frac{1}{q})+d(\frac{1}{r}-\frac{1}{r'})]+\frac{1}{\gamma}(\frac{1}{\alpha}-\frac{1}{2\beta})[\phi(\frac{1}{p}-\frac{1}{2})+\phi(\frac{1}{2}-\frac{1}{q})+1]-1,\nonumber\\
  B=&~\frac{1}{\alpha\gamma}[n(\frac{1}{p}-\frac{1}{q})+d(\frac{1}{r}-\frac{1}{r'})]-1,\nonumber\\
  C=&~\frac{1}{2\beta}[n(\frac{1}{p}-\frac{1}{q})+d(\frac{1}{r}-\frac{1}{r'})]+(\frac{1}{\alpha}-\frac{1}{2\beta})[\phi(\frac{1}{p}-\frac{1}{2})+\phi(\frac{1}{2}-\frac{1}{q})+1]-1,\nonumber\\
  D=&~\frac{1}{\alpha}[n(\frac{1}{p}-\frac{1}{q})+d(\frac{1}{r}-\frac{1}{r'})]-1.
\end{align}

Now, we can state the boundness theorems of restriction operators $\mathcal{P}_\mu^m$.
\begin{thm}\label{mt1}
Given $\alpha,\beta > 0, \gamma \neq 0, 1\leq p\leq 2\leq q\leq\infty, 1\leq r\leq 2\frac{d+1}{d+3}$, $(d,p,q)\neq(1,2,2)$, $m(a,b)=(a^\alpha+b^\beta)^\gamma$, we have for all $f\in \ca{S}(G)$,
\[\parallel \mathcal{P}_\mu^m f \parallel_{L_z^{r'}L_v^q}\leq C_\mu^m \parallel f \parallel_{L_z^rL_v^p},\] with
\beq\label{f-mt1} C_\mu^m \lesssim_{\alpha,\beta,\gamma,p,q,r,n,d}
\left\{\begin{array}{cc}
   \mu^B \qquad & \mu^{\frac{1}{\gamma}(\frac{1}{\alpha}-\frac{1}{2\beta})}\leq 1,\\
   \mu^A \qquad & \mu^{\frac{1}{\gamma}(\frac{1}{\alpha}-\frac{1}{2\beta})}> 1.
\end{array}\right.\eeq
\end{thm}

\emph{Remark:}
\begin{itemize}
  \item
  The theorem can be described in several cases relative to parameters $\alpha/2\beta, \,\gamma$, and also $\mu$ , which can be seen from the following table \[
  \left.\begin{array}{c|c|c|c}

    \gamma & \alpha/2\beta & \mu & \text{sharp exponent of } C_\mu^m\\
    \hline
    >0& <1& >1    & A\\ \cline{3-4}
      &   & \leq1 & B\\ \cline{2-4}
      & >1& >1    & B\\ \cline{3-4}
      &   & \leq1 & A\\ \cline{2-4}
      & =1& \multicolumn{2}{c}{\qquad~~ A=B}\\
    \hline
    <0& <1& >1    & B\\ \cline{3-4}
      &   & \leq1 & A\\ \cline{2-4}
      & >1& >1    & A\\ \cline{3-4}
      &   & \leq1 & B\\ \cline{2-4}
      & =1& \multicolumn{2}{c}{\qquad~~ A=B}
  \end{array}\right.\]
  \item The theorem contains the inhomogeneous operator --- full Laplacian $\Delta_G$, when $\alpha=\beta=\gamma=1$ (also homogeneous ones like $\ca{L}^2-\Delta_G$).
  \item The mix-norm bound cover the uniform-norm bound when $p=q'=r$, especially that on H-type groups, when the exponent function is degenerated to an easy form. Actually, from the Clifford algebra of H-type groups, we have a dimension relation $d<2n$ \cite{kr83harmonicanalysis}, which tells $p=q'=r\le 2\frac{d+1}{d+3}<2\frac{2n+1}{2n+3}$, a critical point for $A$ and $B$, more precisely, $1/p-1/2>s^*$, see (\ref{phi}), and
      \[A=\frac{1}{\gamma}(\frac{n}{\alpha}+\frac{d}{2\beta})(\frac{2}{p}-1)-1, \quad    B=\frac{n+d}{\alpha\gamma}(\frac{2}{p}-1)-1,\] which coincides with the result of \cite{ls13}, just like the case for the sub-Laplacian, when the result of \cite{cc12} coincides with that of \cite{lw11} on H-type groups although slightly different arguments are applied.
\end{itemize}

\begin{thm}\label{mt2}
Given $\alpha,\beta > 0, \gamma\neq 0, 1\leq p\leq 2\leq q\leq\infty, 1\leq r\leq 2\frac{d+1}{d+3}$, $(d,p,q)\neq(1,2,2)$, $m(a,b)=(1+a^\alpha+b^\beta)^\gamma$, we have for all $f\in \ca{S}(G)$,
\[\parallel \mathcal{P}_\mu^m f \parallel_{L_z^{r'}L_v^q}\leq C_\mu^m \parallel f \parallel_{L_z^rL_v^p},\]  with
\beq\label{f-mt2} C_\mu^m \lesssim_{\alpha,\beta,\gamma,p,q,r,n,d}
\left\{\begin{array}{cll}
    \mu^B \qquad & \mu\rightarrow \infty^{sign\gamma}, &\mu^{\frac{1}{\gamma}(\frac{1}{\alpha}-\frac{1}{2\beta})}\leq 1,\\
    \mu^A \qquad & \mu\rightarrow \infty^{sign\gamma}, &\mu^{\frac{1}{\gamma}(\frac{1}{\alpha}-\frac{1}{2\beta})}> 1, \\
    |1-\mu|^D \qquad & \mu\rightarrow e^{\frac{sign\gamma}{\infty}}, &|1-\mu|^{\frac{1}{\alpha}-\frac{1}{2\beta}}\leq 1,\\
    |1-\mu|^C \qquad & \mu\rightarrow e^{\frac{sign\gamma}{\infty}}, &|1-\mu|^{\frac{1}{\alpha}-\frac{1}{2\beta}}> 1.
\end{array}\right.
\eeq
\end{thm}
\emph{Remark:} \\
\begin{itemize}
  \item We have a similar table as last theorem(we give the case $\gamma<0$, which we care more about):
\[
\left.
  \begin{array}{c|c|c}
    \alpha/2\beta & \mu & \text{sharp exponent of } C_\mu^m\\
    \hline
    <1& \rightarrow0+ & A\\\cline{2-3}
      & \rightarrow1- & D\\
      \hline
    >1& \rightarrow0+ & B\\\cline{2-3}
      & \rightarrow1- & C\\
      \hline
    =1& \rightarrow0+ & A=B\\\cline{2-3}
      & \rightarrow1- & C=D
  \end{array}
\right..
\]

  \item The new approximating situation $\mu\rightarrow 1\pm$ is similar, as we get a similar control of $\mu_k$ and $\mu_k'$. Our results includes many useful operators  like the resolvents $(I+\ca{L})^{-1}$ or $(I+\Delta_G)^{-1}$.
\end{itemize}

\section{Proof of the Main Results}\label{proof of main results}
The following important sharp estimate due to H. Koch and F. Ricci \cite{kr07}, about the $L^p\rightarrow L^2$ bound of the spectral projector of twisted Laplacian for $1\le p\le 2$, is critical in our proof,
\beq\label{kr-twisted spectral projector bound}
\|\Lambda_k\|_{L^p(\mathbb{C}^n)\rightarrow L^2(\mathbb{C}^n)} \lesssim_{p,n} (2k+n)^{\phi(\frac{1}{p}-\frac{1}{2})}.
\eeq
\subsection{Series Bound for General $m$}
\begin{lem}\label{l-pq bound}
Let $\Lambda_k^\lambda$, given in Proposition \ref{proposition}, be the spectral projection operator on the Heisenberg group $\mathbb{H}^n$, then for $1\le p\le 2\le q\le\infty$, we have
\beq\label{f-pqbound}\parallel\Lambda_k^\lambda\parallel_{L^p(\mathbb{C}^n)\rightarrow L^q(\mathbb{C}^n)}\lesssim_{p,q,n} \lambda^{n(\frac{1}{p}-\frac{1}{q}-1)}(2k+n)^{\phi(\frac{1}{p}-\frac{1}{2})+\phi(\frac{1}{2}-\frac{1}{q})}.
\eeq
\end{lem}
\begin{proof}
From duality and the projection property of $\Lambda_k=\Lambda_k^1$, we see $\Lambda_k=\Lambda_k^2=\Lambda_k^*\Lambda_k$, and the sharp $L^p\rightarrow L^2$ estimate (\ref{kr-twisted spectral projector bound}) for $1\leq p\leq 2$ gives the following $L^p\rightarrow L^q$ estimate for general exponents $1\leq p\leq 2\leq q\leq \infty$,
\[\parallel\Lambda_k^1\parallel_{L^p(\mathbb{C}^n)\rightarrow L^q(\mathbb{C}^n)}\lesssim (2k+n)^{\phi(\frac{1}{p}-\frac{1}{2})+\phi(\frac{1}{2}-\frac{1}{q})}.\]
By the definition of $\Lambda_k^\lambda$ and changing variables, the twisted convolution
\begin{align*}
  \Lambda_k^\lambda g (z)&=\int_{\mathbb{C}^n} g(z-w)\varphi_k^\lambda(w)e^{i\frac{\lambda}{2} Imz\cdot\bar{w}} dw\\
   &=\lambda^{-n}\int_{\mathbb{C}^n}g[\lambda^{-\frac{1}{2}}(\lambda^{\frac{1}{2}}z-w)]\varphi_k(w)e^{\frac{i}{2}Im(\lambda^{\frac{1}{2}}z)\cdot\bar{w}} dw\\
   &=\lambda^{-n}\delta_{\lambda^{\frac{1}{2}}}(\delta_{\lambda^{-\frac{1}{2}}}g\times_1\varphi_k^1)(z)\\
   &=\lambda^{-n}\delta_{\lambda^{\frac{1}{2}}}\big(\Lambda_k^1(\delta_{\lambda^{-\frac{1}{2}}}g)\big)(z),
\end{align*} where we use dilation $\delta_\lambda g(\cdot)=g(\lambda\,\cdot)$.
So we have
\begin{align*}
               \parallel\Lambda_k^\lambda g\parallel_{L^q(\mathbb{C}^n)}
               &\lesssim \lambda^{-n(1+\frac{1}{q})}\parallel \Lambda_k^1 \parallel_{L^p(\mathbb{C}^n)\rightarrow L^q(\mathbb{C}^n)}\parallel \delta_{\lambda^{-\frac{1}{2}}}g\parallel_{L^p(\mathbb{C}^n)}\\
              & \lesssim \lambda^{n(\frac{1}{p}-\frac{1}{q}-1)}\parallel \Lambda_k^1\parallel_{L^p(\mathbb{C}^n)\rightarrow L^q(\mathbb{C}^n)}\parallel g\parallel_{L^p(\mathbb{C}^n)}\\
              & \lesssim \lambda^{n(\frac{1}{p}-\frac{1}{q}-1)}(2k+n)^{\phi(\frac{1}{p}-\frac{1}{2})+\phi(\frac{1}{2}-\frac{1}{q})}\parallel g\parallel_{L^p(\mathbb{C}^n)},
            \end{align*}
and the hidden ignored constant is dependent of $p,q,n$ from (\ref{kr-twisted spectral projector bound}). Therefore, we get the expected bound of $\Lambda_k^\lambda$.
\end{proof}
\begin{thm}\label{t-series bound}
Given $1\leq p\leq 2\leq q\leq\infty, 1\leq r\leq 2\frac{d+1}{d+3}$ and ``proper" $m(\cdot,\cdot)$, then for all $f\in \ca{S}(G)$,
\[\parallel \mathcal{P}_\mu^m f\parallel_{L_z^{r'}L_v^q}\leq C_\mu^m \parallel f \parallel_{L_z^rL_v^p},
\]with
\beq\label{f-series bound}
C_\mu^m \lesssim_{p,q,r,n,d} \sum_{k=0}^\infty \mu_k'\, \mu_k^{n(\frac{1}{p}-\frac{1}{q})+d(\frac{1}{r}-\frac{1}{r'})-1}(2k+n)^{\phi(\frac{1}{p}-\frac{1}{2})+\phi(\frac{1}{2}-\frac{1}{q})}.
\eeq
If $m(\cdot,\cdot)$ is good enough, the sharp constant is finite.
\end{thm}
\begin{proof}
 Using the relation between the spectral projection of the $\lambda\omega$-twisted Laplacian on $G$ and that on Heisenberg Group $\mathbb{H}^n$ in Proposition \ref{proposition}, we have
 \[\|\Pi_k^{\lambda\omega}\|_{L^p\rightarrow L^q}=|A_\omega|^{\frac{1}{p}-\frac{1}{q}}\|\Lambda_k^\lambda\|_{L^p\rightarrow L^q},\]
 with $|A_\omega|=|B_\omega|^{-1/2}$. As $B_\omega$ is non-degenerate for all $\omega$ and the function $|B_\omega|$ is continous with regards to $\omega$ on the unit sphere, we can assume $|A_\omega|\sim 1$, so from (\ref{f-pqbound}) in last Lemma \ref{l-pq bound}, we have
\beq\label{f-Pi bound}
 \|\Pi_k^{\lambda\omega}\|_{L^p\rightarrow L^q}\lesssim \lambda^{n(\frac{1}{p}-\frac{1}{q}-1)}(2k+n)^{\phi(\frac{1}{p}-\frac{1}{2})+\phi(\frac{1}{2}-\frac{1}{q})}.
\eeq
Denote by\footnote{In abuse of notation.} $\langle\cdot,\cdot\rangle$ the dual action of two functions respectively in two dual $L^p$ and $L^{p'}$ spaces or mix-norm spaces on $G$, i.e.,
\[\langle f,g\rangle= \iint_G \bar{f} g dVdZ.\]
By the formula (\ref{restriction operator formula}) of the restriction operator $\ca{P}_\mu^m$ in Theorem \ref{t-spectral decomposition}, changing integral orders, and then using orderly the H\"{o}der inequality, $L^p \rightarrow L^q$ bound (\ref{f-Pi bound}) of $\Pi_k^{\mu_k\omega}$ , Cauchy-Schwartz inequality, Minkovski inequality (glancing at the exponent $p,q'\leq 2$), and finally the Tomas-Stein theorem, we get for any $f,g\in \ca{S}(G)$,
\begin{align*}
\Big|\langle\mathcal{P}_\mu^mf,g\rangle\Big|&\lesssim \sum_{k=0}^\infty\mu_k'\mu_k^{n+d-1}\int_{S^{d-1}}\bigg|\Big\langle(\Pi_k^{\mu_k\omega}\circ\fr{F}_\fr{z})f(V,\mu_k\omega),g(V,Z)e^{i\mu_{k}\omega(Z)}\Big\rangle\bigg| d\sigma(\omega)\\
&\lesssim \sum_{k=0}^\infty\mu_k'\mu_k^{n+d-1}\int_{S^{d-1}}\|(\Pi_k^{\mu_k\omega}\circ\fr{F}_\fr{z})f(V,\mu_k\omega)\|_{L^q_v} \|\fr{F}_\fr{z}g(V, \mu_k\omega)\|_{L^{q'}_v} d\sigma(\omega)\\
&\lesssim \sum_{k=0}^\infty \mu_k'\mu_k^{n+d-1}\|\Pi_k^{\mu_k\omega}\|_{L^p\rightarrow L^q}\|\fr{F}_\fr{z}f(V,\mu_k\omega)\|_{L^p_vL^2_\omega}\|\fr{F}_\fr{z}g(V, \mu_k\omega)\|_{L^{q'}_vL^2_\omega}\\
&\lesssim \sum_{k=0}^\infty \mu_k'\mu_k^{n+d-1}\|\Pi_k^{\mu_k\omega}\|_{L^p\rightarrow L^q}\|\mu_k^{-d}f(V,\mu_k^{-1}Z)\|_{L^r_zL^p_v} \|\mu_k^{-d}g(V,\mu_k^{-1}Z)\|_{L^r_zL^{q'}_v} \\
&\lesssim \sum_{k=0}^\infty \mu_k'\mu_k^{n+d(1-\frac{2}{r'})-1}\|\Pi_k^{\mu_k\omega}\|_{L^p\rightarrow L^q}\|f\|_{L^r_zL^p_v} \|g\|_{L^r_zL^{q'}_v}\\
&\lesssim \sum_{k=0}^\infty \mu_k'\mu_k^{n(\frac{1}{p}-\frac{1}{q})+d(\frac{1}{r}-\frac{1}{r'})-1}(2k+n)^{\phi(\frac{1}{p}-\frac{1}{2})+\phi(\frac{1}{2}-\frac{1}{q})}\|f\|_{L^r_zL^p_v} \|g\|_{L^r_zL^{q'}_v}.
\end{align*}
By duality, we have proved the bound (\ref{f-series bound}) in the theorem.
\end{proof}

\subsection{$\mu$-Dependent Bound for Two Special Classes of $m$}
Now, with the series bound control (\ref{f-series bound}) associated to general proper functional calculus in Theorem \ref{t-series bound}, we are going to get more sophisticated $\mu$-dependent control for two special cases of functionals $(\mathcal{L}^\alpha+(-\Delta_\fr{z})^\beta)^\gamma$ and $(1+\mathcal{L}^\alpha+(-\Delta_\fr{z})^\beta)^\gamma$ with $\alpha, \beta>0, \gamma\neq0$. In short, we come to prove our main results: Theorem \ref{mt1} and Theorem \ref{mt2}.

\begin{flushleft}\textbf{Proof of Theorem \ref{mt1}:}\end{flushleft}
\begin{proof}
 For theorem \ref{mt1}, we consider operators $(\mathcal{L}^\alpha+(-\Delta_\fr{z})^\beta)^\gamma$, associated to $m(a,b)=(a^\alpha+b^\beta)^\gamma$. Then we have the following easy estimates for $\mu_k$, the solution $\lambda$ of equation $\big((2k+n)\lambda\big)^\alpha + \lambda^{2\beta}=\mu^{\frac{1}{\gamma}}$,
 \beq\label{estimate of mu_k}
 \left\{
 \begin{array}{cl}
    \mu_k & <  \qquad\, \min\{\mu^{\frac{1}{2\beta\gamma}},(2k+n)^{-1}\mu^{\frac{1}{\alpha\gamma}}\},\\
    |\mu_k'| &\sim_{\alpha,\beta,\gamma} ~ \mu^{-1}\mu_k.
 \end{array}
 \right.\eeq
 so from (\ref{f-series bound}) in Theorem \ref{t-series bound}, we have
\begin{align}\label{p-two parts}
  C_\mu^m & \lesssim \sum_{k=0}^\infty \mu^{-1}\, \mu_k^{n(\frac{1}{p}-\frac{1}{q})+d(\frac{1}{r}-\frac{1}{r'})}(2k+n)^{\phi(\frac{1}{p}-\frac{1}{2})+\phi(\frac{1}{2}-\frac{1}{q})}\nonumber\\
  &\lesssim \Big(\sum_{2k+n\leq \mu^{\frac{1}{\gamma}(\frac{1}{\alpha}-\frac{1}{2\beta})}}+\sum_{2k+n\geq \mu^{\frac{1}{\gamma}(\frac{1}{\alpha}-\frac{1}{2\beta})}}\Big)\ldots\nonumber\\  &=I_1+I_2 \nonumber\\
  &=I.
  \end{align}
We consider in two cases\footnote{The two cases can also be divided into several more detailed cases, see table in first term of the remark following Theorem \ref{mt1}.} :
\begin{itemize}
\item[\emph{Case 1.}]
When $\mu^{\frac{1}{\gamma}(\frac{1}{\alpha}-\frac{1}{2\beta})}\leq 1$.\\
In this Case, the first term in (\ref{p-two parts}) can be discarded, so after inserting the estimate (\ref{estimate of mu_k}), we have
\begin{align*}
  I&=I_2 \\
  &=\mu^{-1+\frac{1}{\alpha\gamma}[n(\frac{1}{p}-\frac{1}{q})+d(\frac{1}{r}-\frac{1}{r'})]}\\
  &\quad\times\sum_{2k+n\geq\mu^{\frac{1}{\gamma}(\frac{1}{\alpha}-\frac{1}{2\beta})}} (2k+n)^{\phi(\frac{1}{p}-\frac{1}{2})+\phi(\frac{1}{2}-\frac{1}{q})-[n(\frac{1}{p}-\frac{1}{q})+d(\frac{1}{r}-\frac{1}{r'})]}\\
  &\lesssim \mu^{\frac{1}{\alpha\gamma}[n(\frac{1}{p}-\frac{1}{q})+d(\frac{1}{r}-\frac{1}{r'})]-1}\\
  &=\mu^B.
\end{align*}
Actually, in order to derive the last inequality, it suffices to check the exponent of the power series. First, we denote\footnote{$p_*$ is a critical point as $\frac{1}{p_\ast}-\frac{1}{2}=s^\ast$.}  $p_\ast=2\frac{2n+1}{2n+3}$, and the exponent $\nu=\phi(\frac{1}{p}-\frac{1}{2})+\phi(\frac{1}{2}-\frac{1}{q})-[n(\frac{1}{p}-\frac{1}{q})+d(\frac{1}{r}-\frac{1}{r'})]$. Using $1\leq r \leq 2\frac{d+1}{d+3}$, we check it in four cases corresponding to the piecewise function $\phi$ in (\ref{phi}):\\
\begin{itemize}
\item[a.]
  $p\leq p_\ast, q\geq p_\ast'$.\\
  $\nu=-1-d(\frac{1}{r}-\frac{1}{r'})\leq -1-\frac{2d}{d+1}\leq -2 < -1.$
\item[b.]
  $p\leq p_\ast, q\leq p_\ast'$.\\
  $\nu=-(n+\frac{1}{2})(\frac{1}{2}-\frac{1}{q})-\frac{1}{2}-d(\frac{1}{r}-\frac{1}{r'})\leq -\frac{1}{2}-\frac{2d}{d+1}\leq -\frac{3}{2}<-1.$
\item[c.]
  $p\geq p_\ast, q\geq p_\ast'$.\\
  This case is equivalent to item b. .
\item[d.]
  $p\geq p_\ast, q\leq p_\ast'$.\\
  $\nu=-(n+\frac{1}{2})(\frac{1}{p}-\frac{1}{q})-d(\frac{1}{r}-\frac{1}{r'})\leq -(n+\frac{1}{2})(\frac{1}{p}-\frac{1}{q})-\frac{2d}{d+1}\leq -1$, and
  $\nu<-1$ unless $d=1,r=1,p=q=2$, which is just the bad endpoint case on the Heisenberg group.
\end{itemize}

\item[\emph{Case 2.}]
When $\mu^{\frac{1}{\gamma}(\frac{1}{\alpha}-\frac{1}{2\beta})}> 1$.\\
In this case, we can assume $\mu^{\frac{1}{\gamma}(\frac{1}{\alpha}-\frac{1}{2\beta})}> n$ and need to estimate both of the two terms in (\ref{p-two parts}): after inserting the estimate (\ref{estimate of mu_k}), we have
\begin{align*}
 I_1&=\mu^{-1+\frac{1}{2\beta\gamma}[n(\frac{1}{p}-\frac{1}{q})+d(\frac{1}{r}-\frac{1}{r'})]}\sum_{2k+n\leq\mu^{\frac{1}{\gamma}(\frac{1}{\alpha}-\frac{1}{2\beta})} }(2k+n)^{\phi(\frac{1}{p}-\frac{1}{2})+\phi(\frac{1}{2}-\frac{1}{q})}\\
    &\lesssim \mu^{\frac{1}{2\beta\gamma}[n(\frac{1}{p}-\frac{1}{q})+d(\frac{1}{r}-\frac{1}{r'})]+\frac{1}{\gamma}(\frac{1}{\alpha}-\frac{1}{2\beta})[\phi(\frac{1}{p}-\frac{1}{2})+\phi(\frac{1}{2}-\frac{1}{q})+1]-1}\\
  &=\mu^A.\\
  I_2&=\mu^{-1+\frac{1}{\alpha\gamma}[n(\frac{1}{p}-\frac{1}{q})+d(\frac{1}{r}-\frac{1}{r'})]}\\
  &\quad\times\sum_{2k+n\geq\mu^{\frac{1}{\gamma}(\frac{1}{\alpha}-\frac{1}{2\beta})} }(2k+n)^{\phi(\frac{1}{p}-\frac{1}{2})+\phi(\frac{1}{2}-\frac{1}{q})-[n(\frac{1}{p}-\frac{1}{q})+d(\frac{1}{r}-\frac{1}{r'})]},\\
  &\text{\;checking the convergence of the series as before in \emph{Case 1}, then}\\
  &\lesssim \mu^{\frac{1}{\alpha\gamma}[n(\frac{1}{p}-\frac{1}{q})+d(\frac{1}{r}-\frac{1}{r'})]+\frac{1}{\gamma}(\frac{1}{\alpha}-\frac{1}{2\beta})\{\phi(\frac{1}{p}-\frac{1}{2})+\phi(\frac{1}{2}-\frac{1}{q})-[n(\frac{1}{p}-\frac{1}{q})+d(\frac{1}{r}-\frac{1}{r'})]+1\}-1}\\
  &=\mu^A.
\end{align*}
Here, we need to check the exponent of the series for $I_1$. We denote the exponent by $\nu_1=\phi(\frac{1}{p}-\frac{1}{2})+\phi(\frac{1}{2}-\frac{1}{q})$, then $\nu_1>-1$, as the worst case is that $p\rightarrow p_*$ and $q\rightarrow p_*'$, then $\nu_1=-s^*>-1$.
\end{itemize}
Then from the formulas of the exponents in (\ref{exponent}),  (\ref{f-mt1}) and hence Theorem (\ref{mt1}) is proved.
\end{proof}

\begin{flushleft}\textbf{Proof of Theorem \ref{mt2}:}\end{flushleft}
\begin{proof}
For theorem \ref{mt2}, we consider operators $(1+\mathcal{L}^\alpha+(-\Delta_\fr{z})^\beta)^\gamma$, associated to $m(a,b)=(1+a^\alpha+b^\beta)^\gamma$. We may assume $\gamma <0, \, \mu \in (0,1)$ and $\mu_k$ is the solution of equation $1+\big((2k+n)\lambda\big)^\alpha+\lambda^{2\beta}=\mu^{\frac{1}{\gamma}}$, which is easily seen to be strictly decreasing relative to $\mu$. We naturally consider two boundary cases
\[
\left\{
\begin{array}{cl}
 \mu\rightarrow 0+ & \mu_k\rightarrow \infty, \\
 \mu\rightarrow 1- & \mu_k\rightarrow 0+.
\end{array}\right.\]
First we have similar bound  $\mu_k \leq \min\{(\mu^{\frac{1}{\gamma}}-1)^{\frac{1}{2\beta}},(2k+n)^{-1}(\mu^{\frac{1}{\gamma}}-1)^{\frac{1}{\alpha}}\}$, or given in specific cases,
\beq\label{estimate of mu_k II}\mu_k \lesssim_{\gamma} \left\{\begin{array}{cl}
                   \min\{\mu^{\frac{1}{2\beta\gamma}},(2k+n)^{-1}\mu^{\frac{1}{\alpha\gamma}}\} & \mu\rightarrow 0+, \\
                   \min\{(1-\mu)^{\frac{1}{2\beta}},(2k+n)^{-1}(1-\mu)^{\frac{1}{\alpha}}\}\quad & \mu\rightarrow 1-.
                 \end{array}\right.
\eeq
By the decreasing of $\mu_k$ corresponding to not only $\mu$ but also $k$, we also have
 \beq\label{estimate of mu_k'}|\mu_k'|\backsim \left\{\begin{array}{cl}
                      \mu^{-1}\mu_k & \mu\rightarrow 0+, \\
                      (1-\mu)^{-1}\mu_k  & \mu \rightarrow 1-.
                    \end{array}\right.
 \eeq
From the last two estimates of $\{\mu_k,\mu_k'\}$, which is similar to that in the proof of Theorem \ref{mt1}, we can repeat the proof there, and so similar is the form of the conclusions in two main theorems.
For case $\mu\rightarrow 0+$, the estimate of $\{\mu_k,\mu_k'\}$  is absolutely the same as that in Theorem \ref{mt1}, so is the proof process. For case $\mu\rightarrow 1-$, the process is similar.
By (\ref{f-series bound}) in Theorem \ref{t-series bound} and (\ref{estimate of mu_k II},\ref{estimate of mu_k'}), we have
         \begin{align}\label{two parts II}
          C_\mu^m
          &\lesssim\sum_{k=0}^\infty(1-
          \mu)^{-1}\mu_k^{n(\frac{1}{p}-\frac{1}{q})+d(\frac{1}{r}-\frac{1}{r'})}(2k+n)^{\phi(\frac{1}{p}-\frac{1}{2})+\phi(\frac{1}{2}-\frac{1}{q})}\nonumber\\
          &\lesssim \Big(\sum_{2k+n\leq (1-\mu)^{\frac{1}{\alpha}-\frac{1}{2\beta}}}+\sum_{2k+n\geq (1-\mu)^{\frac{1}{\alpha}-\frac{1}{2\beta}}}\Big)\ldots\nonumber\\
          &=I_1+I_2\nonumber\\
          &=I.
         \end{align}
Again we discussed it in two cases:
\begin{itemize}
\item[\emph{Case 1.}]
When $(1-\mu)^{\frac{1}{\alpha}-\frac{1}{2\beta}}\leq 1$.\\
Still the first term in (\ref{two parts II}) can be omitted, then
\begin{align*}
  I&=I_2 \\
  &=(1-\mu)^{-1+\frac{1}{\alpha}[n(\frac{1}{p}-\frac{1}{q})+d(\frac{1}{r}-\frac{1}{r'})]}\\
  &\quad\times\sum_{2k+n\geq(1-\mu)^{\frac{1}{\alpha}-\frac{1}{2\beta}} } (2k+n)^{\phi(\frac{1}{p}-\frac{1}{2})+\phi(\frac{1}{2}-\frac{1}{q})-[n(\frac{1}{p}-\frac{1}{q})+d(\frac{1}{r}-\frac{1}{r'})]}\\
  &\text{checking the convergence of the series, then}\\
  &\lesssim (1-\mu)^{\frac{1}{\alpha}[n(\frac{1}{p}-\frac{1}{q})+d(\frac{1}{r}-\frac{1}{r'})]-1}\\
  &=(1-\mu)^D.
\end{align*}
The exponent of the series is the same as that in the proof of Theorem \ref{mt1}, i.e., it equals $\nu$ and $<-1$, wiping off the bad endpoint in the Heisenberg case.\\
\item[\emph{Case 2.}]
When $(1-\mu)^{\frac{1}{\alpha}-\frac{1}{2\beta}}> 1$.\\
There are also two terms to estimate.
\begin{align*}
 I_1&=(1-\mu)^{-1+\frac{1}{2\beta}[n(\frac{1}{p}-\frac{1}{q})+d(\frac{1}{r}-\frac{1}{r'})]}\sum_{2k+n\leq(1-\mu)^{\frac{1}{\alpha}-\frac{1}{2\beta}} }(2k+n)^{\phi(\frac{1}{p}-\frac{1}{2})+\phi(\frac{1}{2}-\frac{1}{q})}\\
  &\lesssim (1-\mu)^{-1+\frac{1}{2\beta}[n(\frac{1}{p}-\frac{1}{q})+d(\frac{1}{r}-\frac{1}{r'})]+(\frac{1}{\alpha}-\frac{1}{2\beta})[\phi(\frac{1}{p}-\frac{1}{2})+\phi(\frac{1}{2}-\frac{1}{q})+1]}\\
  &=(1-\mu)^C,\\
  I_2&=(1-\mu)^{-1+\frac{1}{\alpha}[n(\frac{1}{p}-\frac{1}{q})+d(\frac{1}{r}-\frac{1}{r'})]}\\
  &\quad\times\sum_{2k+n\geq(1-\mu)^{\frac{1}{\alpha}-\frac{1}{2\beta}} }(2k+n)^{\phi(\frac{1}{p}-\frac{1}{2})+\phi(\frac{1}{2}-\frac{1}{q})-[n(\frac{1}{p}-\frac{1}{q})+d(\frac{1}{r}-\frac{1}{r'})]},\\
  &\text{\;checking the convergence of the series as before, then}\\
  &\lesssim (1-\mu)^{\frac{1}{\alpha}[n(\frac{1}{p}-\frac{1}{q})+d(\frac{1}{r}-\frac{1}{r'})]+(\frac{1}{\alpha}-\frac{1}{2\beta})\{\phi(\frac{1}{p}-\frac{1}{2})+\phi(\frac{1}{2}-\frac{1}{q})-[n(\frac{1}{p}-\frac{1}{q})+d(\frac{1}{r}-\frac{1}{r'})]+1\}-1}\\
  &=(1-\mu)^C.
\end{align*}

\end{itemize}
Then from (\ref{exponent}), (\ref{f-mt2}) and therefore Theorem (\ref{mt2}) is proved.
\end{proof}

\bibliographystyle{amsalpha}
\bibliography{reference}

\end{document}